\documentclass[bezier,amstex, oneside,reqno]{amsart}
\usepackage{amsmath, amssymb, amsthm, verbatim, euscript, amscd, textcomp}
\usepackage[dvips]{graphicx}
\usepackage{epsfig}
 \usepackage[all,cmtip]{xy}
\usepackage{color}
 \usepackage[fleqn]{cases}

\def\ie{\emph{i.e., }}

\def\R{\mathbb R}
\def\Z{\mathbb Z}

\def\T{\mathbb T}

\newtheorem*{theorem}{Theorem}
\newtheorem{prop}{Proposition}[section]
\newtheorem*{add}{Addendum to Proposition~\ref{prop22}}

\newtheorem{cor}[prop]{Corollary}
\newtheorem{lemma}[prop]{Lemma}
 \theoremstyle{remark}
\newtheorem{remark}[prop]{Remark}

\theoremstyle{remark}

\begin{document}
\author{ F. Thomas Farrell and Andrey Gogolev$^\ast$}
\title[Expanding endomorphisms]{Examples of expanding endomorphisms on fake tori}
\thanks{$^\ast$Both authors were partially supported by NSF grants. The second named author also would like to acknowledge the support provided by Dean's Research Semester Award.}
\begin{abstract}
We construct expanding endomorphisms on smooth manifolds that are
homeomorphic to tori yet have exotic underlying PL-structures.
\end{abstract}
\date{}
 \maketitle

\section{Introduction}
Let $M$ be a closed smooth manifold.
Recall that a smooth map $f\colon M\to M$ is called an {\it expanding endomorphism} if $M$ admits a Riemannian metric $\|\cdot\|$ such that 
$\|D_f(v)\| >\|v\|$ for all non-zero tangent vectors $v$.

Shub \cite{Sh} proved that an expanding endomorphism of a closed manifold $M$ is topologically conjugate to an affine expanding endomorphism of an infranilmanifold if and only if the fundamental group $\pi_1(M)$ contains a nilpotent subgroup of finite index. Franks~\cite{Fr} showed that if $M$ admits an expanding endomorphism then $\pi_1(M)$ has polynomial growth. Finally, in 1981, Gromov \cite{Gr} completed classification by showing that any finitely generated group of polynomial growth contains a nilpotent
subgroup of finite index. Hence any expanding endomorphism is topologically conjugate to an affine expanding endomorphism of an infranilmanifold. In particular, any manifold that supports an expanding endomorphism is homeomorphic to an infranilmanifold.

Farrell and Jones \cite{FJ} showed that any connected sum $\T^d\#\Sigma^d$ of the standard $d$-dimensional torus $\T^d=\R^d/\Z^d$ and a homotopy sphere $\Sigma^d$ admits an expanding endomorphism. When $\Sigma^d$ $(d\geq7)$ is not diffeomorphic to the standard sphere $\mathbb S^d$, this constraction provides an example of a manifold that admits an expanding endomorphism and is not diffeomorphic to any infranilmanifold. However, $\T^d\#\Sigma^d$ is always PL homeomorphic to the standard torus $\T^d$ via the Alexander trick.

A $d$-dimensional smooth manifold $M$ is called a {\it fake torus} if $M$ is homeomorphic to $\T^d$ but not PL homeomorphic to $\T^d$. In this paper we obtain the first examples of fake tori that admit expanding endomorphisms.

\begin{theorem}
For any $d\geq 7$ there exists a $d$-dimensional fake torus $M$ that admits an expanding endomorphism. 
\end{theorem}

\begin{remark}
Lee and Raymond~\cite{LR} showed that any isomorphism between the fundamental groups of a pair of infranilmanifolds is induced by a diffeomorphism. It follows that a fake torus is not PL homeomorphic to any infranilmanifold.
\end{remark}

\section{Fake tori}
Let $n\geq6$ and let $\T^n=\R^n/\Z^n$ be the $n$-torus.
Consider the standard $3^n$-sheeted self-covering map $\pi\colon \T^n\to\T^n$ given by 
$$
\pi(x)=3x\;\;mod\;\;\Z^n.
$$

If $h\colon \T^n\to \T^n$ is a diffeomorphism which induces the identity homomorphism on $\pi_1(\T^n)$ then $h$ lifts through $\pi$; \ie there exists a diffeomorphism $\tilde{h} \colon\T^n\to\T^n$ that makes the following diagram commute
$$
\begin{CD}
\T^n @>\tilde{h}>> \T^n \\
@V{\pi}VV @V{\pi}VV \\
\T^n @>{h}>> \T^n 
\end{CD}
$$
Note that there are exactly $3^n$ such liftings and they are all diffeotopic.
We say that two diffeomorphisms are {\it diffeotopic} if they are isotopic through a smooth path of diffeomorphisms.

\begin{prop} \label{prop21}
For any $n\geq 6$ there exists a diffeomorphism $h\colon \T^n\to \T^n$ such that 
\begin{enumerate}
\item $h$ is topologically pseudo-isotopic to  $id_{\T^n}$;
\item $h$ is not PL pseudo-isotopic to $id_{\T^n}$;
\item $\tilde{h}$ is diffeotopic to $h$;
\item $h^2=h\circ h$ is diffeotopic to  $id_{\T^n}$.
\end{enumerate}
\end{prop}

Recall that given a diffeomorphism $h\colon \T^n\to \T^n$ the {\it mapping torus} of $h$ is defined as
$$
M_h =[0, 1] \times \T^n/(1,x)\sim(0, h(x)).
$$
\begin{prop} \label{prop22}
A diffeomorphism $h\colon \T^n\to \T^n$, $n\geq 5$, is PL pseudo-isotopic to $id_{\T^n}$ if and only if the mapping torus $M_h$ is PL homeomorphic to $\T^{n+1}$. 
\end{prop}
We immediately obtain the following result.

\begin{cor} \label{cor23}
If $h\colon \T^n\to \T^n$, $n\geq 6$, is a diffeomorphism given by Proposition~\ref{prop21} then the mapping torus $M_h$ is a fake torus.
\end{cor}

\begin{add}
In fact, the assignment
$
h\mapsto M_h
$
gives a bijection between the smooth pseudo-isotopy classes of diffeomorphisms $h$ of $\T^n$ such that $h_\#=id_{\pi_1(\T^n)}$ and the smooth structures $\theta$ on $\T^{n+1}$ such that the inclusion map
$\sigma\colon \T^n\times \{0\}\to(\T^n\times S^1, \theta)$
is a smooth embedding.
\end{add}

\begin{remark}\label{rem24}
By smoothing theory~\cite[page 194]{KS} (also see~\cite{Ru}) these structures naturally correspond to the subgroup of those elements $\varphi\in[\T^{n+1}, Top/O]$ for which $\varphi \circ\sigma\in[\T^n, Top/O]$ is the zero element, \ie $\varphi\circ\sigma$ null-homotopic.
\end{remark}

\begin{proof}[Proof of Proposition~\ref{prop22}]
It is straightforward to see that if $h$ is PL pseudo-isotopic to $id_{\T^n}$ then $M_h$ is PL homeomorphic to $\T^{n+1}$.

To see the other implication, let $F\colon\T^{n+1}\to M_h$ be a PL homeomorphism. Since $M_h$ is homeomorphic to $\T^{n+1}$, the induced map $h_\#$ on $\pi_1(\T^n)$ must be identity. Hence the fundamental groups of $\T^{n+1}=\T^n\times S^1$ (we view $\T^{n+1}$ as a mapping torus of $id_{\T^n}$) and $M_h$ are ``canonically'' identified (up to specifying a cross section to the bundle projection $M_h\to S^1$). By precomposing $F$ with a diffeomorphism of $\T^{n+1}$ we can assume that $F_\#\colon \pi_1(\T^{n+1})\to \pi_1(M_h)$ is the canonical identification.

Recall that $M_h$ can be viewed as the total space of a $\T^n$-bundle over $S^1$. Denote by $N\subset M_h$ a distinguished fiber of this bundle. Then, since $F_\#$ is the canonical identification, we have
\begin{equation}\label{eq21}
i_\#(\pi_1(\T^n\times\{0\}))=F_\#^{-1}(\pi_1(N)),
\end{equation}
where $i\colon \T^n\times\{0\}\to \T^n\times S^1$ is an inclusion of a fiber.

Now consider the trivial cobordism $\T^{n+1}\times[0,1]$. View the top boundary $\T^{n+1}\times\{1\}$ as the trivial bundle $\T^n\times S^1$. Using $F$ we pull back the fibration of $M_h$ to obtain a PL $\T^n$-bundle over $S^1$ on the bottom boundary $\T^{n+1}\times \{0\}$. Then $F^{-1}(N)$ is a distinguished fiber on the bottom. Because of the equation~(\ref{eq21}) and the fact that the Whitehead group $Wh(\pi_n(\T^{n+1}))$ vanishes we can apply the fibering theorem of~\cite{F} (which requires $n\geq 4$) to $\T^{n+1}\times [0, 1]$. This way we obtain a PL fibering of $\T^{n+1}\times [0,1]$ over $S^1$ that extends the fiberings of the top and the bottom boundaries. In particular, we obtain a PL submanifold $W^{n+1}\subset \T^{n+1}\times [0,1]$ such that $\partial^-W^{n+1}=F^{-1}(N)$ and $\partial^+W^{n+1}=(\T^n\times\{0\})\times\{1\}$, where $\T^n\times\{0\}\subset \T^{n+1}$ is a fiber of the trivial fibering of $\T^{n+1}\times\{1\}$ (see Figure~\ref{fig1}).

\begin{figure}[htbp]
\begin{center}
\includegraphics{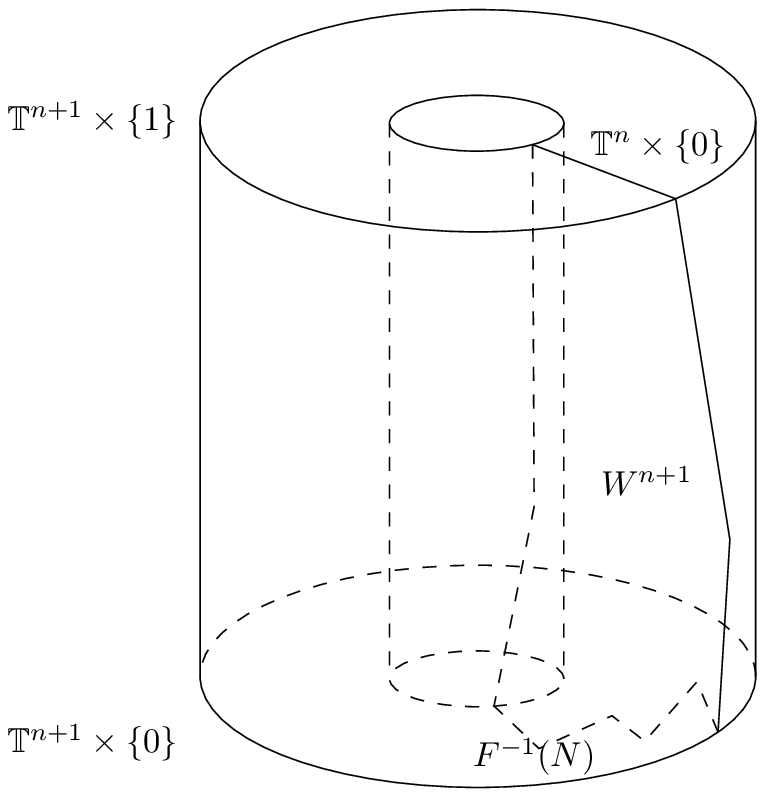}
\end{center}
 \caption{}
\label{fig1}
\end{figure}

It is easy to check that $W^{n+1}$ is an $h$-cobordism. Since $Wh(\pi_1(\T^n))$ vanishes the $s$-cobordism theorem (which requires $n\geq 5$) applies and we conclude that $W^{n+1}$ is PL homeomorphic to the product $\T^n\times [0,1]$. Cutting $\T^{n+1}\times[0,1]$ open along $W^{n+1}$ yields a new relative $h$-cobordism $\mathcal W^{n+2}$ between $\T^{n+1}$ cut along $F^{-1}(N)$ and $\T^{n+1}$ cut along $\T^n\times\{0\}$. Again, since $Wh(\pi_n(\T^n))=0$, this $h$-cobordism is a PL product extending the product structure on $W^{n+1}$. Therefore we obtain a homeomorphism $G\colon \T^{n+1}\times[0,1]\to \T^{n+1}\times[0,1]$ such that 
$$
G|_{\T^{n+1}\times\{1\}}=id_{\T^n} \;\;\mbox{and} \;\;G((\T^n\times\{0\})\times\{0\})=F^{-1}(N).
$$
Let $F_1=F\circ G|_{\T^{n+1}\times\{0\}}$. Then we have $F_1(\T^n\times\{0\})=N\subseteq M_h$.
Cutting $\T^{n+1}$ and $M_h$ open along $\T^n\times\{0\}$ and $N$ respectively yields a PL pseudo-isotopy between $h\circ \psi$ and $\psi$, where $\psi=F_1|_{\T^n\times\{0\}}$. Precomposing this pseudo-isotopy with 
$$\psi^{-1}\times id\colon \T^n\times[0,1]\to\T^n\times[0,1]$$
 yields the desired PL pseudo-isotopy between $h$ and $id_{\T^n}$.
\end{proof}

\begin{proof}[Proof of Proposition~\ref{prop21}]
Let $m=n-1$. Consider the following diagram
$$
\T^{m+1}=\T^{m-2}\times\T^3\xrightarrow{pr}\T^3\xrightarrow{\rho}\mathbb S^3\xrightarrow{\alpha}Top/O,
$$
where $pr$ is the projection to the $\T^3$ factor, $\rho$ is a degree one map and $\alpha\in \pi_3(Top/O)\simeq\Z_2$ is the generator. Let $\varphi$ be the composite map
$$
\varphi=\alpha\circ\rho\circ pr.
$$
Also let $\sigma\colon \T^m\times\{0\}\to\T^{m+1}=\T^m\times S^1$ be the inclusion map. Clearly $\varphi\circ \sigma$ is null-homotopic. Hence, by Remark~\ref{rem24} (in which $m$ replaces $n$) the smoothing $\theta$ of $\T^{m+1}$ corresponding to $\varphi$ determines a diffeomorphism $g\colon\T^m\to\T^m$ such that the mapping torus of $g$ is $(\T^{m+1}, \theta)$. Now we define 
$$h=g\times id\colon \T^{n-1}\times S^1\to\T^{n-1}\times S^1,$$
 and proceed to verify properties 1-4 of the proposition. 

{\it Verification of property 1.} By the definition of $h$ the induced map $h_\#$ on the fundamental group is identity. Hence, by the work of Hsiang and Wall~\cite{HW}, $h$ is topologically pseudo-isotopic to $id_{\T^n}$.

{\it Verification of property 2.} Assume to the contrary that $h$ is PL pseudo-isotopic to $id_{\T^n}$. Then, by Proposition~\ref{prop22}, the mapping torus $M_h=M_g\times S^1$ is PL homeomorphic to $\T^{n+1}=\T^n\times S^1$. As before, we can arrange that the PL homeomorphism between $M_g\times S^1$ and $\T^n\times S^1$ induces the canonical identification of the fundamental groups. We pass to the infinite cyclic covering to obtain a PL homeomorphism $F\colon M_g\times \R\to\T^n\times \R$. By cutting $\T^n\times \R$ along $F(M_g\times\{0\})$ and $\T^n\times\{t\}$ for a sufficiently large $t$ we obtain a PL $h$-cobordism between $M_g$ and $\T^n$. Since $Wh(\pi_1(\T^n))$ vanishes, we conclude that $M_g$ is PL homeomorphic to $\T^n$. 

Recall that by the work of Hsiang and Wall~\cite{HW} homotopic homeomorphisms of $\T^n$ are topologically pseudo-isotopic. Using this fact together with PL ``smoothing" theory of Kirby and Siebenmann~\cite{KS}, we conclude that
$$
\gamma \circ\varphi\colon\T^n\to Top/PL=K(\Z_2, 3)
$$
is null-homotopic. Here $\gamma$ is the canonical map $Top/O\to Top/PL$.

Now let $a$ be the generator of $H^3(Top/PL, \Z_2)$. Then the cohomology class $(\gamma \circ \varphi)^*(a)\in H^3(\T^n,\Z_2)$ vanishes. Since $(\rho\circ pr)^*\colon H^3(\mathbb S^3, \Z_2)\to H^3(\T^n,\Z_2)$ is clearly monic, we obtain that 
\begin{equation}\label{eq_vanish}
(\gamma\circ \alpha)^*(a)=0
\end{equation}
But $\gamma_\#\colon \pi_3(Top/O)\to\pi_3(Top/PL)$ is an isomorphism. Hence $\gamma\circ\alpha$ is the generator of $\pi_3(Top/PL)=\pi_3(K(\Z_2,3))$ because of our choice of $\alpha$. Therefore $(\gamma\circ\alpha)^*\colon H^3(Top/PL, \Z_2)$ $\to H^3(\mathbb S^3, \Z_2)$ is also monic. Thus~(\ref{eq_vanish}) implies that $a=0$ yielding a contradiction.

{\it Verification of property 3.} Recall that $\theta$ is the smooth structure on $\T^n$ which corresponds to homotopy class of $\varphi$. The $3^n$-sheeted covering map $\pi\colon\T^n\to\T^n$ induces a smooth structure $\omega$ on $\T^n$; namely, the smooth structure such that $\pi\colon (\T^n, \omega)\to (\T^n,\theta)$ is a smooth codimension zero immersion. This smooth structure corresponds to the homotopy class of the continuous map $\varphi\circ \pi\colon \T^n\to Top/O$. This map is clearly the same map as the composite map 
$$
\T^{n-3}\times\T^3\xrightarrow{pr}\T^3\xrightarrow{\pi'}\T^3\xrightarrow{\rho}\mathbb S^3\xrightarrow{\alpha}Top/O,
$$
where $\pi'\colon \T^3\to\T^3$ is given by $x\mapsto(3x \;\;\mbox{mod}\;\; \Z^3)$.
Since $\rho\circ\pi' \colon \T^3\to\mathbb S^3$ has odd degree (namely 27), the map $\alpha\circ \rho\circ\pi'$ is homotopic to $\alpha\circ\rho$. This is because the canonical map $\gamma\colon Top/O\to Top/PL=K(\Z_2, 3)$ is $7$-connected and $\T^3$ is $3$-dimensional. Therefore $\varphi\circ\pi$ is homotopic to $\varphi$ and $(\T^n, \omega)$ is smoothly concordant to $(\T^n, \theta)$.

Now let $\tilde{g}\colon\T^{n-1}\to\T^{n-1}$ be a lifting of $g$ through the covering map $x\mapsto(3x\;\; \mbox{mod}\;\; \Z^{n-1})$. Then, by the Addendum to Proposition~\ref{prop22}, $\tilde{g}$ is smoothly pseudo-isotopic to $g$. Let $F\colon \T^{n-1}\times [0,1]\to\T^{n-1}\times[0,1]$ be such a pseudo-isotopy; \ie let $F$ be a diffeomorphism such that
$$
F|_{\T^{n-1}\times\{0\}}=g \;\; \mbox{and}\;\; F|_{\T^{n-1}\times\{1\}}=\tilde{g}.
$$
Since $S^1$ has zero Euler characteristic, the product formula for concordances~\cite[Proposition on p. 18]{H} applied to $F$, yields a diffeotopy between $h=g\times id_{S^1}$ and $\tilde{h}=\tilde{g}\times id_{S^1}$.

{\it Verification of property 4.} Let $\tau\colon S^1\to S^1$ be the double covering map given by $x\mapsto (2x \;\;\mbox{mod}\;\; \Z)$ and let $\Omega$ be the lifting of $\theta$, \ie the smooth structure on $\T^n$ such that $id\times\tau\colon(\T^{n-1}\times S^1, \Omega)\to(\T^{n-1}\times S^1, \theta)$ is a smooth codimension zero immersion. This smooth structure corresponds to the homotopy class of the continuous map $\varphi\circ (id\times\tau)\colon\T^n\to Top/O$. This map is clearly the same as the composite map
$$
\T^{n-3}\times\T^3\xrightarrow{pr}\T^3=\T^2\times S^1\xrightarrow{id\times\tau}\T^2\times S^1\xrightarrow{\rho}\mathbb S^3\xrightarrow{\alpha}Top/O.
$$
Since $\rho\circ(id\times\tau)\colon\T^3\to\mathbb S^3$ has even degree (namely 2), the map $\alpha\circ\rho\circ(id\times\tau)$ is null-homotopic. Hence $\varphi\circ(id\times\tau)$ is also null-homotopic, and $(\T^n, \Omega)$ is smoothly concordant to $\T^n$ equipped with its natural smooth structure. 

A standard argument shows that the mapping torus of $g^2$ is smoothly concordant to $(\T^n, \Omega)$. We conclude, by the Addendum to Proposition~\ref{prop22}, that $g^2$ is smoothly pseudo-isotopic to $id_{\T^{n-1}}$. Therefore, by arguing as we did in verifying property 3, we see that $h^2=(g\times id_{S^1})^2= g^2\times id_{S^1}$ is diffeotopic to $id_{\T^{n-1}} \times id_{S^1} = id_{\T^n}$.
\end{proof}

\section{Construction of the expanding endomorphism}

Consider the diffeomorphism $h\colon \T^n\to\T^n$ given by Proposition~\ref{prop21}. Let $M_h$ be the mapping torus of $h$. Manifold $M_h$ is a fake torus by Proposition~\ref{prop22}. For each $k\geq 1$ and $m\geq 1$, we will define self-covering maps $p_k, q_m\colon M_h\to M_h$. Our strategy is to obtain an expanding endomorphism of $M_h$ by composing $p_k$ and $q_m$ for sufficiently large $k$. Roughly speaking, $p_k$ will be expanding with respect to $x\in \T^n$, that is, along the fibers and $q_m$ will be expanding with respect to $t$, that is, transversely to the fibers.

\begin{subsection}{Construction of covering map $p_k$.}
Recall that $\pi\colon \T^n\to\T^n$ is given by $x\mapsto (3x \;\;\mbox{mod}\;\; \Z^n)$. Let $h_0=h$ and let $h_1\colon\T^n\to\T^n$ be some lift of $h_0$ through $\pi$. By property $3$ of Proposition~\ref{prop21}, $h_0$ is diffeotopic to $h_1$. Hence there exists a diffeotopy $\varphi_1\colon [0,1]\times\T^n\to\T^n$ such that $\varphi_1(0, \cdot)=id_{\T^n}$ and $\varphi_1(1, \cdot)=h_1^{-1}\circ h_0$.

Next we define a sequence of liftings $\{h_i; i\geq 0\}$ inductively. Assume that for some $i\geq 1$ we have defined a diffeomorphism $h_{i-1}\colon \T^n\to\T^n$ and its lifting $h_i$. Also assume that we have a diffeotopy $\varphi_i\colon[0,1]\times\T^n\to\T^n$ that connects $id_{\T^n}$ to $h_i^{-1}\circ h_{i-1}$. Then, by the Lifting Lemma, there exists a unique diffeotopy $\varphi_{i+1}$ such that $\varphi_{i+1}(0, \cdot)=id_{\T^n}$ and the diagram
$$
\begin{CD}
\T^n @>{\varphi_{i+1}(t,\cdot)}>> \T^n \\
@V{\pi}VV @V{\pi}VV \\
\T^n @>{\varphi_i(t,\cdot)}>> \T^n 
\end{CD}
$$ 
commutes for every $t\in[0,1]$. Define
$$
h_{i+1}=h_i\circ\varphi_{i+1}(1,\cdot)^{-1}.
$$
It is easy to see that $h_{i+1}$ is a lifting of $h_i$. Also note that each $h_i$ is a (particular) lifting of $h$ through the covering map $\pi^i$.

Next we describe two auxiliary manifolds $N_k$ and $M_{h_k}$. Manifold $M_{h_k}$ is simply the mapping torus of $h_k$, \ie
$$
M_{h_k} =[0, 1] \times \T^n/(1,x)\sim(0, h_k(x)).
$$
Manifold $N_k$ is a {\it multiple mapping torus} of the sequence of diffeomorphisms $h=h_0$, $h^{-1}_{k-1}\circ h_k$,  $h^{-1}_{k-2}\circ h_{k-1}, \ldots$ $h_0^{-1}h_1$ defined in the following way
$$
N_k=\bigsqcup^{k}_{i=1}\left[\frac{i}{k+1},\frac{i+1}{k+1}\right]\times\T^n\left/
\begin{array}{ll}
(1, x)\sim(0,h(x)),& \\[6pt]
 \left(r\left[\frac{i}{k+1},\frac{i+1}{k+1}\right], x\right)\sim\left(l\left[\frac{i+1}{k+1},\frac{i+2}{k+1}\right],h^{-1}_{k-i-1}(h_{k-i}(x))\right), \\[6pt]
  1\leq i\leq k-1
\end{array}\right.
$$
Here if $[a,b]$ is an interval then $l[a,b]=a$ and $r[a,b]=b$.

Now we will construct three auxiliary maps $H_k$, $F_k$ and $P_k$ whose domains and ranges are indicated in the diagram below
$$
\xymatrix{
N_k \ar[r]^{F_k} & M_{h_k} \ar[ld]^{P_k}\\
M_h\ar[u]^{H_k} &}
$$
We define these maps by the following formulae

\begin{multline*}
H_k(t,x)=
\begin{cases}
(t, \varphi_{k-i}\left((k+1)(t-\frac{i}{k+1}),x)\right), &\mbox{if} \;\;\frac{i}{k+1}\leq t\leq\frac{i+1}{k+1}, i=0, \ldots k-1
\\
(t, x), &\mbox{if} \;\;\frac{k}{k+1}\leq t\leq 1
\end{cases}\\[6pt]
\shoveleft{F_k(t, x)=
\begin{cases}
(t,x), &\mbox{if} \;\;0\leq t\leq\frac{1}{k+1}
\\
(t, h_k^{-1}(h_{k-i}(x))), &\mbox{if} \;\;\frac{i}{k+1}\leq t\leq\frac{i+1}{k+1}, i=1,\ldots k
\end{cases}
}
\\[6pt]
\shoveleft{P_k(t,x)=(t, 3^kx\;\;\mbox{mod}\;\;\Z^n).\hfill}
\end{multline*}
Since $h_k$ is a lifting of $h$ with respect to $\pi^k\colon\T^n\to\T^n$, the last formula indeed gives a well defined $3^{kn}$-sheeted covering map $P_k\colon M_{h_k}\to M_h$. The fact that the first two formulae give well defined diffeomorphisms can be checked directly using Figures~\ref{fig2} and~\ref{fig3}.

\begin{figure}[htbp]
\begin{center}
\includegraphics{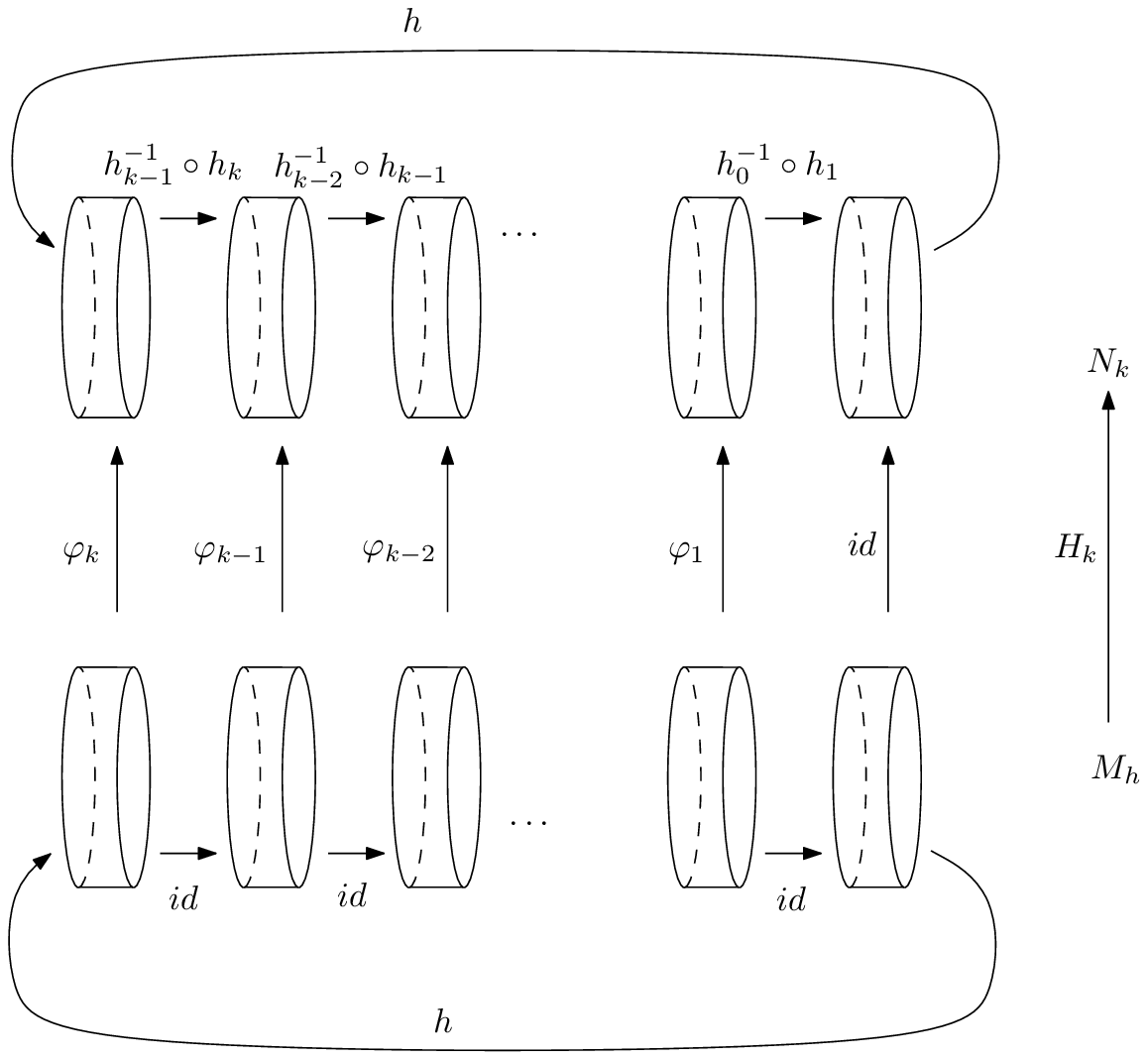}
\end{center}
 \caption{}
\label{fig2}
\end{figure}

\begin{figure}[htbp]
\begin{center}
\includegraphics{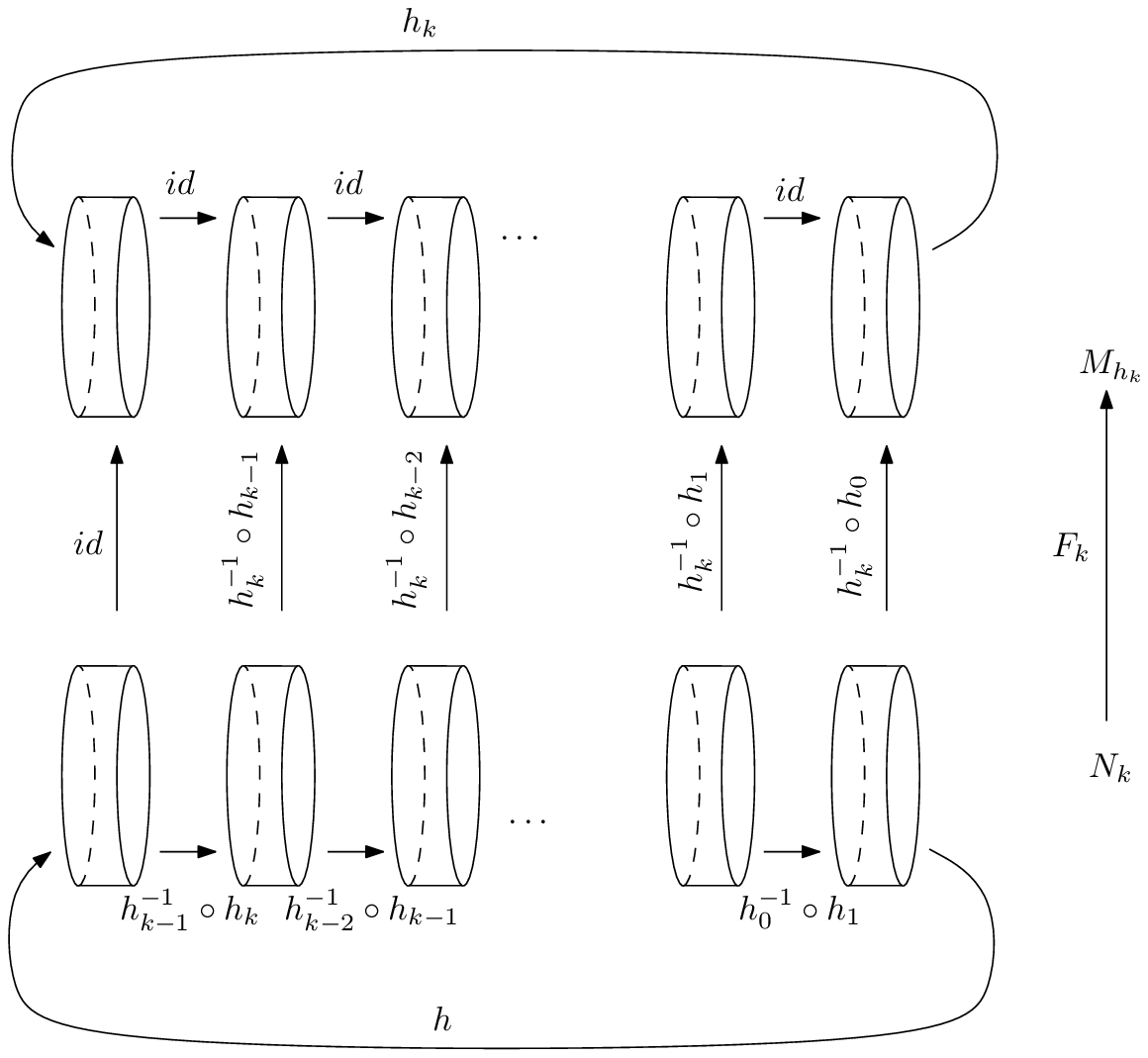}
\end{center}
 \caption{}
\label{fig3}
\end{figure}

Finally define
$$
p_k=P_k\circ F_k\circ H_k.
$$
It is clear from our definitions that $p_k$ has the form $p_k(t, x)=(t, \rho_{k, t}(x))$, where each $\rho_{k, t}$ is a $3^{kn}$-sheeted self-covering map.
\end{subsection}

\begin{subsection}{The expanding property of $p_k$.}
We equip $M_h$ with a Riemannian metric in the following way. Let $g_0$ be the standard flat metric on $\T^n$ and let $g_1=h^*g_0$. For every $t\in[0,1]$ we let $g_t=(1-t)g_0+tg_1$. For each point $(t,x)\in M_h$ we equip the tangent space $T_{t, x}M_h$ with the scalar product $g_t(x)+dt^2$.
Clearly this defines a Riemannian metric on $M_h$ which we denote by $G$. We write $\|v\|_G$ for the norm of a vector $v\in TM_h$. If a vector $v\in TM_h$ is tangent to a torus fiber of $M_h$ then we say that $v$ is a {\it vertical vector} and we write $v\in T^{\|}M_h$.

\begin{lemma}\label{lemma4}
For any $\lambda>1$ there exists $k\geq1$ such that
$$
\|D{p_k}(v)\|_G\geq \lambda\|v\|_G,  \;\;\mbox{for any}\quad v\in T^{\|}M_h.
$$
\end{lemma}
\begin{proof}
For each $k\geq1$ define the Riemannian metric on the fibers of $M_{h_k}$ by letting
$$
g_{k,t}=(F_k\circ H_k|_{\{t\}\times \T^n})_* g_t.
$$
Then $G_k(t,x)=g_{k,t}(x)+dt^2$ is a Riemannian metric on $M_{h_k}$. We will write $\|v\|_{G_k}$ for the $G_k$-norm of a vector $v\in TM_{h_k}$. We can also equip each fiber of $M_h$ and $M_{h_k}$ with a flat metric $g_0$. Indeed, recall that both mapping tori $M_h$ and $M_{h_k}$ can be identified with $[0,1)\times \T^n$. Hence each fiber is identified with the flat torus $(\T^n, g_0)$. We will write $\|v\|_0$ for the $g_0$-norm of a vertical vector $v\in T^\|M_h$. Same notation $\|v\|_0$ will be used for $v\in T^\|M_{h_k}$.

The proof of the lemma is based on the following facts.
\begin{enumerate}
\item There exists a constant $c>0$ such that $c\leq \frac{\|v\|_G}{\|v\|_0}\leq\frac{1}{c}$ for all $v\in T^\|M_{h}$;
\item $\|DP_k(v)\|_0=3^k\|v\|_0$ for all $v\in T^\|M_{h_k}$;
\item There exists a constant $C>0$ such that $\|D(F_k\circ H_k|_{\{t\}\times\T^n})(v)\|_0\geq C\|v\|_0$ for all $k\geq 1$, $t\in[0,1)$ and $v\in T^\|M_h$.
\end{enumerate}

The first fact is due to compactness of $M_h$. The second one follows from the definition of $P_k$. To see the third fact recall that the diffeomorphisms $F_k\circ H_k|_{\{t\}\times \T^n}$ are the compositions of at most three diffeomorphisms of $\T^n$ of the form $h_i$, $h_i^{-1}$ and $\varphi_j(t,\cdot)$, where $i,j\in\{1,2,\ldots k\}$ and $t\in [0,1]$.
The latter diffeomorphisms are liftings of $h$, $h^{-1}$ and $\varphi_1(t,0)$, $t\in[0,1]$, respectively, through the various homothetic self-covering maps $\pi^s$, $s\ge 1$. Hence the conorms of differentials of these diffeomorphisms are bounded uniformly in $i, j$ and $t$ from below.

For a non-zero $v\in T^\|M_h$, we have
\begin{multline*}
\displaystyle\frac{\|D(P_k\circ F_k\circ H_k)(v)\|_G}{\|v\|_G}=
 \displaystyle\frac{\|D(P_k\circ F_k\circ H_k)(v)\|_G}{\|D(P_k\circ F_k\circ H_k)(v)\|_0}\cdot
 \displaystyle\frac{\|D(P_k\circ F_k\circ H_k)(v)\|_0}{\|v\|_0}\cdot
 \displaystyle\frac{\|v\|_0}{\|v\|_G}\\[10pt]
 \geq c^2\displaystyle\frac{\|D(P_k\circ F_k\circ H_k)(v)\|_0}{\|v\|_0}=c^23^k\displaystyle\frac{\|D(F_k\circ H_k)(v)\|_0}{\|v\|_0}\geq c^23^kC.
\end{multline*}
Recall that $c$ and $C$ do not depend on $k$. Choose $k$ so that $c^23^kC>\lambda$. The lemma follows.
\end{proof}
\end{subsection}

\begin{subsection}{Construction of covering map $q_m$}
Pick an integer $m\geq 1$. We will define auxiliary manifolds $\bar{M}_{h, m}$, ${M}'_{h, m}$ and $\widetilde{M}_{h,m}$ together with maps between them as indicated on the diagram below
$$
\xymatrix{
\bar{M}_{h, m} \ar[r]^{S_m} & M'_{h, m} \ar[r]^{T_m} & \widetilde{M}_{h, m}\ar[ld]^{Q_m}\\
& M_h\ar[lu]^{R_m} }
$$
Manifold $\bar{M}_{h,m}$ is a ``long mapping torus" 
$$\bar{M}_{h,m}=[0, 2m+1]\times\T^n/(2m+1,x)\sim(0,h(x)).$$ Manifolds $M'_{h,m}$ and $\widetilde{M}_{h,m}$ are multiple mapping tori defined as follows
$$
M'_{h,m}=\bigsqcup^{2m+1}_{i=1}[i-1,i]\times\T^n\left/
\begin{array}{ll}
(2m+1, x)\sim(0,h(x)),& \\[6pt]
 (r[i-1, i], x)\sim(l[i,i+1],x), & \hbox{if $i$ is odd}, 1\leq i\leq 2m-1,\\[6pt]
  (r[i-1, i], x)\sim(l[i,i+1],h^2(x)), & \hbox{if $i$ is even}, 2\leq i\leq 2m
\end{array}\right.
$$
$$
\widetilde{M}_{h,m}=\bigsqcup^{2m+1}_{i=1}[i-1,i]\times\T^n\left/
\begin{array}{ll}
(2m+1, x)\sim(0,h(x)), & \\[6pt]
(r[i-1, i], x)\sim(l[i,i+1],h(x)), & 1\leq i\leq 2m
\end{array}
\right.
$$

Now we define the maps. We set
$$
R_m(t,x)=((2m+1)t, x).
$$
By property $4$ in Proposition~\ref{prop21} there exists a diffeotopy $\psi\colon[0,1]\times\T^n\to\T^n$ such that $\psi(0, \cdot)=id_{\T^n}$ and $\psi(1, \cdot)=h^{-2}$. Define
$$
S_m(t,x)=
\begin{cases}
(t,x) &\;\mbox{if}\;\; i\leq t\leq i+1, i=0, 2, \ldots 2m
\\
(t, \psi(t-i, x)) &\;\mbox{if}\;\; i\leq t\leq i+1, i=1,3, \ldots 2m-1
\end{cases}
$$
and
$$
T_m(t, x)=
\begin{cases}
(t,x) &\;\mbox{if}\;\; i\leq t\leq i+1, i=0, 2, 4, \ldots 2m
\\
(t, h(x)) &\;\mbox{if}\;\; i\leq t\leq i+1, i=1,3, \ldots 2m-1
\end{cases}
$$
\begin{figure}[htbp]
\begin{center}
\includegraphics{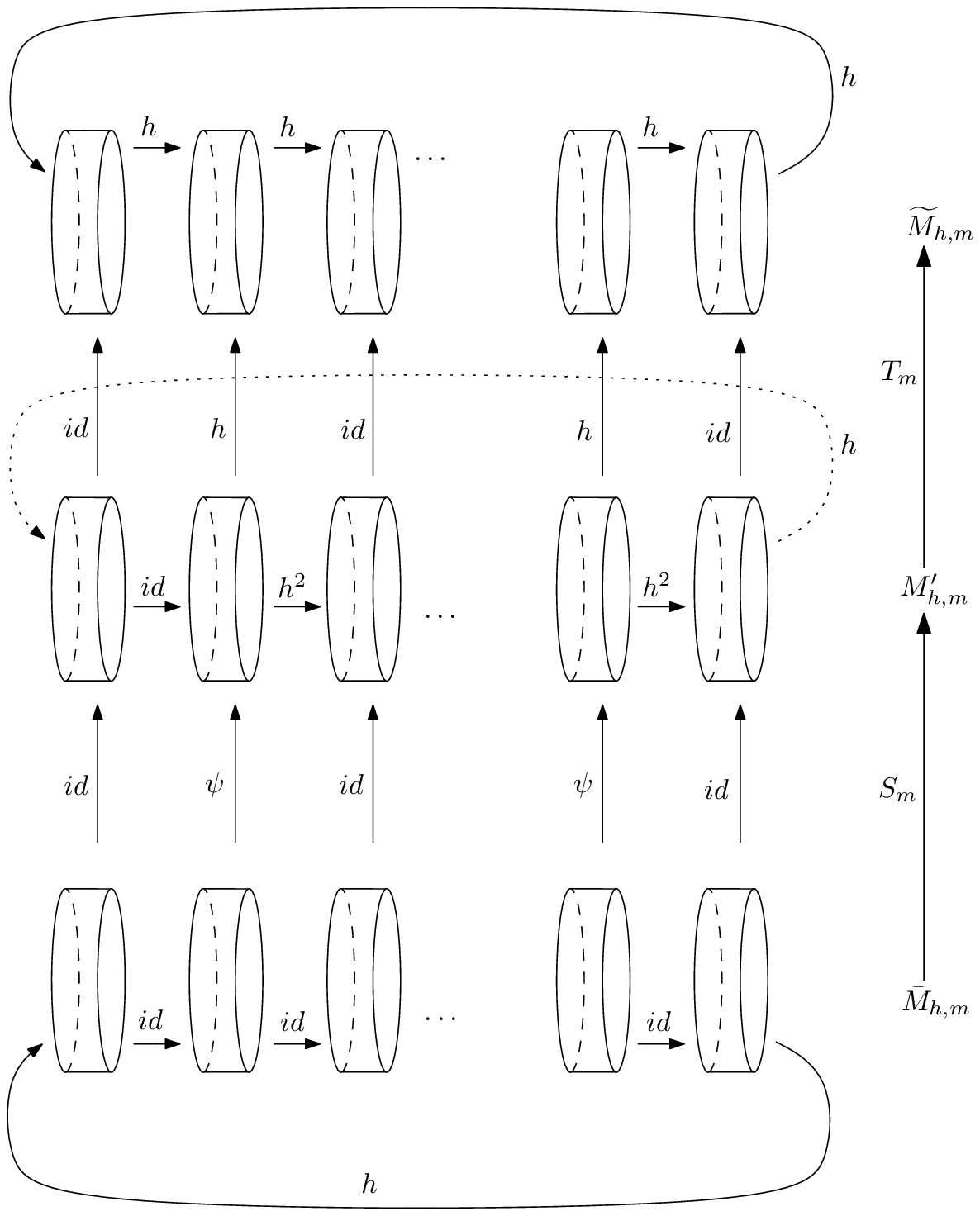}
\end{center}
 \caption{}
\label{fig4}
\end{figure}
With the help of Figure~\ref{fig4} one can check that $S_m$ and $T_m$ are well-defined diffeomorphisms. 

Finally define
$$
Q_m(t,x)=(t\;\;\mbox{mod}\;\; \Z, x)
$$
and 
$$
q_m=Q_m\circ T_m\circ S_m\circ R_m.
$$

It is clear that $q_m$ is $(2m+1)$-sheeted self-covering map of the form
$$
q_m(t,m)=((2m+1)t\;\;\mbox{mod}\;\; \Z, \xi_{m,t}(x)).
$$
\end{subsection}

\begin{subsection}{The expanding map $f$.}
Fix any $m\geq 1$ and consider the $q_m\colon M_h\to M_h$ constructed above. Let
$$
c=\inf_{\substack{v\in T^\|M_h\\ v\neq 0}}\displaystyle\frac{\|D{q_m}(v)\|_G}{\|v\|_G}
$$
Clearly $c>0$.

Now apply Lemma~\ref{lemma4} with $\lambda=\frac{2}{c}$ to obtain an integer $k\geq 1$. Define
$$
f=q_m\circ p_k.
$$
Then $f\colon M_h\to M_h$ is a self-covering map of the form
$$
f(t, x)=((2m+1)t\;\;\mbox{mod}\;\; \Z, \alpha_t(x)).
$$
Note that the subbundle $T^\|M_h$ is $Df$-invariant. 
\begin{lemma}\label{lemma5}
For any $v\in T^\|M_h$
$$
\|Df(v)\|_G\geq 2\|v\|_G.
$$
\end{lemma}
\begin{proof}
$$
\|Df(v)\|_G=\|D{q_m}(D{p_k}(v))\|_G\geq c\|D{p_k}(v)\|_G\geq c\lambda\|v\|_G=2\|v\|_G,
$$
where the last inequality is provided by Lemma~\ref{lemma4}.
\end{proof}

Let $T^\perp M_h$ be the $1$-dimensional orthogonal complement of $T^\|M_h$. Then for any vector $v\in TM_h$ there is a unique decomposition
$$
v=v^\|+v^\perp,
$$
where $v^\|\in T^\|M_h$ and $v^\perp \in T^\perp M_h$. For a vector $v^\perp\in T^\perp M_h$ we have
$$
Df(v^\perp)^\perp=(2m+1)v^\perp
$$
Also, since $M_h$ is compact,  there exists a constant $K>0$ such that 
\begin{equation}\label{eq_estimate}
\|(Dfv^\perp)^\|\|_G\leq K\|v^\perp\|_G.
\end{equation}

Now we are ready to prove that $f$ is expanding via a standard cone argument. We equip $TM_h$ with a Finsler metric $|||\cdot |||$ defined in the following way
$$
|||v|||=\max\left(\frac{1}{K}\|v^\|\|_G, \|v^\perp\|_G\right).
$$
We now show that $f$ is expanding with respect to $|||\cdot |||$.

{\it Case 1:} $\|v^\perp\|_G\geq\frac{1}{K}\|v^\|\|_G$. In this case we have 
$$
|||Df(v)|||\geq\|Df(v)^\perp\|_G=(2m+1)\|v^\perp\|_G=(2m+1)|||v |||.
$$

{\it Case 2:} $\|v^\perp\|_G<\frac{1}{K}\|v^\|\|_G$.
We use~(\ref{eq_estimate}) and Lemma~\ref{lemma5} for the estimate below
\begin{multline*}
|||Df(v)|||\geq\frac{1}{K}\|Df(v)^\|\|_G=\frac{1}{K}\|Df(v^\|)^\|+Df(v^\perp)^\|\|_G\\
\geq 
\frac{1}{K}\|Df(v^\|)^\|\|_G-\frac{1}{K}\|Df(v^\perp)^\|\|_G\geq \frac{2}{K}\|v^\|\|_G-\frac{1}{K}\cdot K\|v^\perp\|_G\\
>
\frac{2}{K}\|v^\|\|_G-\frac{1}{K}\|v^\|\|_G=\frac{1}{K}\|v^\|\|_G=|||v|||.
\end{multline*}

Because $M_h$ is compact we actually have
$$
|||Df(v)|||>\mu|||v||| 
$$
for some $\mu>1$ and all non-zero $v\in TM_h$. Now let $\|\cdot\|$ be a Riemannian metric on $M_h$. Then, because $\|\cdot\|$ and $|||\cdot|||$ are equivalent ($c|||\cdot|||\le\|\cdot\|\le C|||\cdot|||$), the last inequality implies that for a sufficiently large $N$
$$
\|Df^N(v)\|>\|v\|
$$
for all non-zero $v\in TM_h$. Then, using the standard adapted metric construction~\cite{Math}, we can find a Riemannian metric $\|\cdot\|_{ad}$ such that
$$
\|Df(v)\|_{ad}>\|v\|_{ad}
$$
for all non-zero $v$. Hence $f$ is expanding. $\hfill\Box$
\end{subsection}

$~$\\
$~$\\

F.T. Farrell, A. Gogolev

SUNY Binghamton, N.Y., 13902, U.S.A.\\

\end{document}